\documentclass[12pt,a4paper]{amsart}

\usepackage[latin2]{inputenc}
\usepackage[a4paper]{geometry}

\usepackage{graphics}
\usepackage{verbatim}
\usepackage{listings}
\usepackage[mathcal]{eucal}
\usepackage{amsmath}
\usepackage{amssymb}
\usepackage{multicol}
\usepackage{t1enc}
\usepackage{amsthm}

\usepackage{tikz-cd}

\newtheorem{thm}{Theorem}[section]

\newtheorem{cor}[thm]{Corollary}

\theoremstyle{definition}
\newtheorem{dfn}[thm]{Definition}
\newtheorem{exa}{Example}
\newtheorem{rem}{Remark}

\begin{document}

\title{Representations of Leibniz algebras}

\author[A. Fialowski]{Alice Fialowski}
\author[\'E. Zs. Mih\'alka]{\'Eva Zsuzsanna Mih\'alka}

\address{Institute of Mathematics\\
E\"otv\"os Lor\'and University\\
 P\'azm\'any P\'eter s\'et\'any 1/C\\
1117 Budapest, Hungary}

\email{fialowsk@cs.elte.hu}
\email{mihalka@coulson.chem.elte.hu}



\begin{abstract}
In this paper we prove that every irredicuble representation of a
Leibniz algebra can be obtained from irreducible representations of the
semisimple Lie algebra from the Levi deconposition. We also prove that
- in general - for (semi)simple Leibniz algebras it is not true that
a representatiom can be decomposed to a direct sum of irreducible ones.
\end{abstract}

\maketitle

\section{Introduction}

The notion of Leibniz algebra was introduced by A.M. Bloh (\cite{Bloh1,
Bloh2} in the 1960-s, and later rediscovered and developed by J-L. Loday (\cite{Loday})
in 1993. Since then it became very popular, mainly because of its
applications in physics. A number of theorems for Lie algebras were
generalized for Leibniz algebras, like Lie's Theorem, Engel's Theorem,
Cartan's criterium, Levi's Theorem (\cite{Barnes1, Barnes3, Barnes4, O}), but
some other results do not hold for Leibniz algebras. Representations of
Leibniz algebras were introduced in \cite{LP}. Beside that, there is a recent
result on faithful representations (\cite{Barnes2}). 

In this paper we consider representations of semisimple Leibniz
algebras, and study, what can be carried over from the known Theorems
in the Lie case. We prove that every irreducible representation of a
Leibniz algebra $L$ can be obtained from irreducible representations of the
semisimple Lie algebra $S$ from the decomposition $L=S\dot{+}I$, where $I$ is 
the Leibniz kernel, $S$ is a
semisimple Lie algebra which is a subalgebra in $L$. We also prove that
for (semi)simple Leibniz algebras it is not true in general that a
representation can be decomposed to a direct sum of irreducible ones,
by giving a counterexample.

\section{Basic definitions}

For basic definitions and properties of Leibniz algebras we refer to
\cite{Loday, LP, AO, AAO}. 
\begin{dfn}

Let $K$ be a field. The algebra $(L,[\cdot,\cdot])$ is called a
\emph{Leibniz algebra} over $K$, if for every $x,y,z \in L$ we have the Leibniz
identity:

\centerline{$[[x,y],z]=[[x,z],y]+[x,[y,z]]$}

\end{dfn}
Obviously, every Lie algebra is a Leibniz algebra as well.

\begin{dfn}
We call a $K$-linear map $d: L\rightarrow L$ a \emph{derivation}, if

\centerline{$d[x,y]=[dx,y]+[x,dy]\:(x,y \in L).$}
\end{dfn}

Denote by $Der_{K}(L)$ the algebra of all derivations of $L$.

If we define the product as the bracket operation, $Der_{K}(L)$
becomes a Lie algebra. The Leibniz identity in $L$ means that for every
$x\in L$,  $r_x:=[.,x] \in Der_{K}(L)$ is an inner derivation.

\begin{thm}

$Inn(L):=\{r_x|x\in L\}$ forms a Lie algebra and $Inn(L)\triangleleft Der(L)$.

\end{thm}

Because of this property, these Leibniz algebras are also called right
Leibniz algebras. If we define the Leibniz bracket, assuming that the
left multiplication should be a derivation of $L$, then we call such
algebras left Leibniz algebras. (Of course, one can state every
analogous
property for left Leibniz algebras as well.)

\begin{dfn}
The \emph{Leibniz kernel} of a Leibniz algebra $L$ is $I=span\{x^2 | x\in L\}$,
where $x^2=[x,x]$.
\end{dfn}

From $[x+y,x+y]\in I$ we get that for every $x,y\in L$ we have $[x,y]+[y,x]\in I$.
If $char(K)\neq 2$, then $[x,x]=\frac{1}{2}([x,x]+[x,x])$, so $span\{[x,y]+[y,x]|x,y\in L\}=I$.

\begin{thm}
The Leibniz kernel $I$ is a commutative subalgebra, it is also ideal in
$L$, and the factor algebra $L/I$ is a Lie algebra. This is the
smallest ideal in $L$ for which the factor is a Lie algebra. 
\end{thm}

For Lie algebras the square of every element is $0$, so $I=0$.

Note that if $dim(L)\geq 1$, we have $I\neq L$ for the Leibniz kernel
$I$. 

If $L$ is not a Lie algebra ($I\neq 0$), then $I$ is a nontrivial ideal
in $L$, so the definition of simplicity for Lie algebras can not be
applied for Leibniz algebras. Instead, the definition is modified as
follows:

\begin{dfn}
The Leibniz algebra $L$ is \emph{simple}, if $[L,L]\neq I$ and it only
has the following three ideals: $0,I,L$. (Here $0$ and $I$ are not
necessarily different.)
\end{dfn}
\begin{rem}
As for Lie algebras $I=0$, the new definition of simplicity
coincides with the old one.
\end{rem}

If $L$ is a simple Leibniz algebra, then $L/I$ is simple Lie algebra,
but the opposite is not true.

\begin{dfn}
For a Leibniz algebra $(L,[\cdot,\cdot])$ , define the composition
chains of ideals 
\begin{multicols}{2}
$L^1=L,\:L^{k+1}=[L^k,L],\ k\geq1,$

$L^{[1]}=L,\:L^{[n+1]}=[L^{[n]},L^{[n]}],\ n\geq1$,
\end{multicols}
\end{dfn}

\begin{dfn}
The Leibniz algebra $L$ is \emph{solvable}, if there exists an integer
$n\geq 1$ such that $L^{[n]}=0$, and $L$ is \emph{nilpotent}, if there exists an
integer $k\geq 1$ such that $L^{k}=0$.
\end{dfn}

\begin{thm}
For positive integers $i,j$ we have $[L^i,L^j]\subseteq L^{i+j}$. From
this it follows that for every $k\geq 2$,  $L^{[k]}\subseteq
L^{2^{k-1}}$, so every nilpotent Leibniz algebra is solvable.
\end{thm}

For the Leibniz kernel $I$ of a Leibniz algebra, $I^{[2]}=[I,I]=0$, so
$I$ is solvable.\\
If $L$ is simple, then $I\neq [L,L]$. But $[L,L]$ is an ideal in $L$,
and $I\subseteq [L,L]$ for every Leibniz algebra. We get that if $L$ is
simple, then the only possibility is $[L,L]=L$. From this it also
follows that $L^{k}=L^{[k]}=L\ (k\geq 1)$, so $L$ is neither nilpotent nor solvable.

\begin{cor}
If $L$ is finite dimensional, it has a maximal solvable ideal $R$,
which we call the \emph{radical} of $L$. Also, there exists a maximal
nilpotent ideal, containing every nilpotent ideal. We call this the
\emph{nilradical} of $L$, and denote it by $N$. Clearly, $N\leq R$.
\end{cor}

\begin{dfn}
A Leibniz algebra $L$ is \emph{semisimple}, if its maximal solvable
ideal is $I$.
\end{dfn}

Obviously, in every case $I\triangleleft R$. So if $L$ is simple, $R=I$
or $R=L$. Because of the solvability of $R$, $[R,R]\neq R$, which means
$[R,R]\subset R$ is a proper ideal. If $R=L$, then $I\subseteq
[L,L]=[R,R]\subset R=L$. In $L$, except $0,I,L$ there are no other
ideals, so we would get $[L,L]=I$, which contradicts the simplicity of
$L$. So for $L$ we have $R=I$, which means that from simplicity it
follows semisimplicity.

If $L$ is a Lie algebra, again we get that the two definitions of
semisimplicity coincide.

A Leibniz algebra $L$ is semisimple if and only if the
factor algebra $L/I$ is a semisimple Lie algebra.

\begin{dfn}
Let $L$ be a Leibniz algebra, $M$ vector space over the field $K$.
Assume we have two $K$-linear functions:

$\lambda : L\rightarrow \mathfrak{gl}(M)$ and

$\rho : L\rightarrow \mathfrak{gl}(M)$.

Denote $\lambda(x)$ and $\rho(y)$ by $\lambda_x$ and
 $\rho_y$  for every $x,y \in L$. We say that $M$ is a {\it
 representation} of
 $L$ if the following properties are satisfied:

$(1)$ $\rho_{[x,y]}=\rho_y \rho_x - \rho_x \rho_y$,

$(2)$ $\lambda_{[x,y]}=\rho_y \lambda_x - \lambda_x \rho_y$,

$(3)$ $\lambda_{[x,y]}=\rho_y \lambda_x + \lambda_x \lambda_y$, for
every $x,y\in L$.
\end{dfn}

If $M$ is a representation of $L$, then $M$ becomes an $L$-module with
the following $[\cdot,\cdot]: M\times L\rightarrow M$ and
$[\cdot,\cdot]: L\times M\rightarrow M$ products: 
$[m,x]:=\rho_x(m)$ and
$[x,m]:=\lambda_x(m)$ for every $x\in L, m\in M$.

Conversely, for a given $L$-module $M$, we get the representation
$\rho: L\rightarrow \mathfrak{gl}(M)$, $\lambda: L\rightarrow
\mathfrak{gl}(M)$ with

$\rho_x:=[\cdot,x]\ \forall x\in L$,
$\lambda_x:=[x,\cdot]\ \forall x\in L$.\\

Let $L$ be a Leibniz-algebra, $M$ vector space over $K$, and $\rho
,\lambda :L\rightarrow \mathfrak{gl}(M)$ a representation of $L$ ($M$ is
an $L$-bimodule). Denote for $x\in L$,  $ \rho(x)=[.,x]\in
\mathfrak{gl}(M)$ and $\lambda(x)=[x,.] \in \mathfrak{gl}(M)$ the
right and left multiplication by $x$.

\begin{dfn}
Let $\rho_1,\lambda_1:L\rightarrow
\mathfrak{gl}(M_1)$ and $\rho_2,\lambda_2:L\rightarrow
\mathfrak{gl}(M_2)$ be two Leibniz representations. These two
representations are \emph{equivalent}, if
there exists an isomorphism $\varphi:M_1\rightarrow M_2$ such that $\varphi\circ\rho_1(x)=\rho_2(x)\circ\varphi$
$(\forall x\in L)$, and
$\varphi\circ\lambda_1(x)=\lambda_2(x)\circ\varphi$ $(\forall x\in L)$.
\end{dfn}

Define now the map $\psi:\mathfrak{gl}(M_1)\rightarrow \mathfrak{gl}(M_2)$
 as follows: for $f\in\mathfrak{gl}(M_1)$ let $\psi(f):=\varphi\circ
f\circ \varphi^{-1}$. We get the isomorphism
$\mathfrak{gl}(M_1)\rightarrow \mathfrak{gl}(M_2)$ with the
property that the two diagrams commute:

\begin{multicols}{2}
\[
\begin{tikzcd}
L \arrow{r}{\rho_1} \arrow[swap]{dr}{\rho_2} & \mathfrak{gl}(M_1) \arrow{d}{\psi} \\
& \mathfrak{gl}(M_2)
\end{tikzcd}
\]

\[
\begin{tikzcd}
L \arrow{r}{\lambda_1} \arrow[swap]{dr}{\lambda_2} & \mathfrak{gl}(M_1) \arrow{d}{\psi} \\
& \mathfrak{gl}(M_2)
\end{tikzcd}
\]
\end{multicols}

If $L$ is a 
Lie algebra, then its Lie representation $\varphi_1:L\rightarrow \mathfrak{gl}(M_1)$
 becomes a Leibniz-reprezentation with
$\lambda:=\varphi_1$ and $\rho:=-\varphi_1$. (One can also take the
choice
$\lambda:=0$ and $\rho:=-\varphi_1$.)

If the Lie algebra $L$ has two equivalent Lie representations,
$\varphi_1,\varphi_2$ (with the isomorphism $\varphi:M_1\rightarrow M_2$), 
then, using the above method to form Leibniz representations, they
will be equivalent as well (with the isomorphism $\varphi$).

Let $L$ be a Leibniz algebra and
$\rho_i,\lambda_i:L\rightarrow\mathfrak{gl}(M_i)$ $(i=1,2)$ two
Leibniz-representations of $L$. Assume that these representations are
isomorphic (either via
$\psi:\mathfrak{gl}(M_1)\rightarrow\mathfrak{gl}(M_2)$, or via the
isomorphisms $\varphi: M_1\rightarrow M_2$). For a given  $x\in L$
assume that $0\neq v\in M_1$ is an eigenvector of the map $\rho_1(x)$
for the eigenvalue $\alpha$, $\rho_1(x)(v)=\alpha v$. Then
$\varphi(v)\neq 0$ and
$\rho_2(x)(\varphi(v))=\varphi(\rho_1(x)(v))=\varphi(\alpha
v)=\alpha\varphi(v)$, so $\alpha$ is an eigenvalue of $\rho_2(x)$, and $\varphi(v)$ is an eigenvector for $\alpha$.

\begin{dfn}
Let $K$ be a field, $L$ a Leibniz algebra over $K$,
and $V$ a vector space over $K$. We say that the Leibniz
representation $\rho,\lambda :L\rightarrow \mathfrak{gl}(V)$ is
\emph{irreducible}, if $\rho$ and $\lambda$ are irreducible. In other words, if $U\leq V$ is
an invariant subspace of $\rho$ and $\lambda$ ($\forall x\in L$,
$\rho_x(U)\subseteq U$ and $\lambda_x(U)\subseteq U$), then $U=0$ or $U=V$.
\end{dfn}

\begin{dfn}
We say that the representation $\rho,\lambda :L\rightarrow \mathfrak{gl}(V)$ is
 the \emph{direct
sum} of
lower dimensional ones, if $V$ can be written as the direct sum of
$\rho$- and $\lambda$-invariant subspaces $V=V_1\oplus V_2\oplus \dots
\oplus V_k$ such that for every  $x\in L$, $\rho_x(V_i)\subseteq V_i$
and $\lambda_x(V_i)\subseteq V_i$ ($i=1,\dots,k$).
\end{dfn}


If a Leibniz algebra is not semisimple, then already in the
2-dimensional case the study of its representations is very
complicated.

Let us now deal with the semisimple case, and study what can be carried
over from the statements for Lie algebras to Leibniz algebras.

\section{Representations of semisimple Leibniz algebras}

\begin{thm}\label{irred_lie}
Let $L$ be a Leibniz algebra, $M$ vector space and $\lambda, \rho:
L\rightarrow \mathfrak{gl}(M)$ representation of $L$. Assume that
the representation is irreducible, so $M$ does not have any proper nontrivial
subspace, invariant for $\rho(x)$ and $\lambda(x)$ ($x\in L$). Then a Leibniz
representation can be essentially viewed as Lie representation. We mean
that either $\rho+\lambda=0$ (multiplication by $x$ is anticommutative
on $M$), or $\lambda=0$ (and so defining $\varphi:=-\rho$ we get a Lie
representation).
\end{thm}

\begin{proof}
Let $V:=span\{[y,m]+[m,y]|y\in L, m\in M\}$. Then

\centerline{$0\leq V\leq \displaystyle\bigcap_{x\in L}Ker([x,.])\leq M$,}
and $V$ is invariant subspace for $\rho$ and $\lambda$. The first inclusion is
obvious, for the second let $x,y\in L, m\in M$ be arbitrary, then using
the properties of Leibniz representation, we have
$[x,[y,m]+[m,y]]=[[x,y],m]-[[x,m],y]+[[x,m],y]-[[x,y],m]=0$ and we
proved the second inclusion. This also proves the invariance for $\lambda$.

The invariant property for $\rho$ is also clear: let again $x,y\in L, m\in M$
arbitrary.\\
$[[y,m]+[m,y],x]=[[y,x],m]+[y,[m,x]]+[[m,x],y]+[m,[y,x]]=([[y,x],m]+[m,[y,x]])+([y,[m,x]]+[[m,x],y])\in
V$, because $[y,x]\in L$ and $[m,x]\in M$.

We showed that indeed $V$ is an invariant subspace for $\rho$ and $\lambda$. On the
other hand, we assumed that the representation is irreducible, so there are two
possible cases:
\begin{enumerate}
\item $V=0$. Then for a given $y\in L$, $[y,m]+[m,y]=0$ $\forall m\in
M$, so $\lambda(y)+\rho(y)=0$. As $y$ was arbitrary, we get $\lambda+\rho=0$.

\item $V=M$. Then $M=V=\displaystyle\bigcap_{x\in L}Ker([x,.])= M$, so
$\forall x\in L$,  $Ker([x,.])=M$. But this exactly means that for
every
$x\in L$, $\lambda(x)=0$, in other words, $\lambda=0$.\\
Then by the properties of a Leibniz representation\\
($\rho([x,y])=\rho(y)\rho(x)-\rho(x)\rho(y)$), for the linear map
$\varphi:=-\rho$, $\varphi([x,y])=[\varphi(x),\varphi(y)]$, so
$\varphi$ is a Lie-representation.
\end{enumerate}
\end{proof}

\begin{cor}
The existence of an invariant subspace depends on the irreducibility of
$\rho$, so a Leibniz representation $(\rho,\lambda)$ is irreducible exactly
when $\rho$ is irreducible.
\end{cor}

As a consequence, for a Leibniz representation not being a Lie
representation in the above sense, it is necessary that $\rho$ is not
irreducible.

On the other hand, it is easy to show that for a representation not being Lie, it is not
sufficient that $\rho$ is not irreducible. If $dim(M)\geq 2$, then with
the choice $\rho=\lambda=0$ we clearly get a Lie representation, but
because of $\rho=0$, every subspace is invariant, and becuase of
$dim(M)\geq 2$, there exists a proper nontrivial subspace.

We know the following:

\begin{thm}[Theorem Levi] \cite{Barnes1}
Let $L$ be a finite dimensional Leibniz algebra over $K$, where
$char(K)=0$. Let $R\triangleleft L$ its solvable radical. Then there
exists a semisimple subalgebra $S\leq L$ such that $L=S+R$ and $S\cap
R=0$, so $L=S\dot{+}R$. This subalgebra $S$ is a semisimple Lie
algebra.
\end{thm}

Using this Theorem we get the following:

\begin{cor}
Let $L$ be a semisimple Leibniz algebra. Then $L=S\dot{+}I$, where $I$ is
the Leibniz kernel, $S\leq L$ and $S$ is a semisimple Lie algebra.
\end{cor}

Indeed, as $L$ is semisimple, $R=I$. Then $L=S\dot{+}R$ and $S \cong
L/I$, which means $S$ is a semisimple Lie algebra.

\section{Irreducible representations of semisimple Leibniz algebras}

Let us start with an example.

\subsection{Leibniz representations of the algebra $\mathfrak{sl}_2$} 

Computing Leibniz representations of the Lie algebra $\mathfrak{sl}_2$ is
straightforward.

Let $V$ be a finite dimensional complex vector space and denote
$m+1:=dim_{\mathbb{C}}(V)$. Let
$\rho,\lambda:\mathfrak{sl}_2\rightarrow \mathfrak{gl}(V)$ a Leibniz
representation of the algebra $\mathfrak{sl}_2$. By the multiplication
table the necessary conditions are the following:
 
\begin{enumerate}
\item $\rho_h=\rho_f\rho_e-\rho_e\rho_f$
\item $2\rho_e=\rho_h\rho_e-\rho_e\rho_h$
\item $2\rho_f=\rho_f\rho_h-\rho_h\rho_f$
\item $\lambda_h=\rho_f\lambda_e-\lambda_e\rho_f=\lambda_f\rho_e-\rho_e\lambda_f$
\item $\lambda_h=\rho_f\lambda_e+\lambda_e\lambda_f=-\lambda_f\lambda_e-\rho_e\lambda_f$
\item $0=\rho_h\lambda_h-\lambda_h\rho_h=\rho_h\lambda_h+\lambda_h^2$
\item $2\lambda_e=\rho_h\lambda_e-\lambda_e\rho_h=\lambda_h\rho_e-\rho_e\lambda_h$
\item $2\lambda_e=\rho_h\lambda_e+\lambda_e\lambda_h=-\lambda_h\lambda_e-\rho_e\lambda_h$
\item $0=\rho_e\lambda_e-\lambda_e\rho_e=\rho_e\lambda_e+\lambda_e^2$
\item $2\lambda_f=\rho_f\lambda_h-\lambda_h\rho_f=\lambda_f\rho_h-\rho_h\lambda_f$
\item $2\lambda_f=\rho_f\lambda_h+\lambda_h\lambda_f=-\lambda_f\lambda_h-\rho_h\lambda_f$
\item $0=\rho_f\lambda_f-\lambda_f\rho_f=\rho_f\lambda_f+\lambda_f^2$

\end{enumerate}

Let us restrict ourselves to the irreducible case, assuming that  $\rho$
is irreducible. If we concentrate on the first three conditions, we see
that $\rho$ by itself gives a Lie representation of $\mathfrak{sl}_2$,
or more precisely its $(-1)$ multiple. On the other hand, we know that
for every $m$, $\mathfrak{sl}_2$ has a unique (up to equivalence)
$(m+1)$-dimensional irreducible representation, and in an appropriate
basis, we know the matrices
corresponding to the elements $e,f,h$. From this we get that in an
appropriate basis of $V$,  $\rho_e,\rho_f,\rho_h$ has the following
form $(1\leq i,j \leq m+1)$:

\begin{multicols}{2}
\[(\rho_e)_{i,j} = \left\{
\begin{array}{l l}
i(m+1-i), & \quad \mbox{if $j=i+1$}\\
0, & \quad \mbox{if $j\neq i+1$}\\
\end{array} \right. \]

\[(\rho_f)_{i,j} = \left\{
\begin{array}{l l}
-1, & \quad \mbox{if $j=i-1$}\\
0, & \quad \mbox{if $j\neq i-1$}\\
\end{array} \right. \]

\end{multicols}

\begin{multicols}{2}
\[(\rho_h)_{i,j} = \left\{
\begin{array}{l l}
m+2-2i, & \quad \mbox{if $j=i$}\\
0, & \quad \mbox{if $j\neq i$}\\
\end{array} \right. \]

\end{multicols}

In matrix form:

$\rho_e=\begin{pmatrix}
0 & m & 0 & 0 & \cdots & 0 & 0\\
0 & 0 & 2(m-1) & 0 & \cdots & 0 & 0 \\
0 & 0 & 0 & 3(m-2) & \cdots & 0 & 0\\
\vdots & \vdots & \vdots & \ddots & \ddots & \vdots & \vdots\\
0 & 0 & 0 & \cdots & 0 & 2(m-1) & 0\\
0 & 0 & 0 & \cdots & 0 & 0 & m\\
0 & 0 & 0 & \cdots & 0 & 0 & 0
\end{pmatrix}$

$\rho_f=\begin{pmatrix}
0 & 0 & 0 & \cdots & 0 & 0 & 0\\
-1 & 0 & 0 & \cdots & 0 & 0 & 0\\
0 & -1 & 0 & \cdots & 0 & 0 & 0\\
0 & 0 & -1 & \ddots & \vdots & \vdots & \vdots\\
\vdots & \vdots & \vdots & \ddots & 0 & 0 & 0\\
0 & 0 & 0 & \cdots & -1 & 0 & 0\\ 
0 & 0 & 0 & \cdots & 0 & -1 & 0
\end{pmatrix}$

$\rho_h=\begin{pmatrix}
m & 0 & 0 & \cdots & 0\\
0 & m-2 & 0 & \cdots & 0\\
0 & 0 & m-4 & \cdots & 0\\
\vdots & \vdots & \vdots & \ddots & \vdots\\
0 & 0 & 0 & \cdots & -m
\end{pmatrix}$

Let $(\lambda_f)_{i,j}$ be the element of the $i$-th row
and $j$-th column of $\lambda_f$. Easy to compute that
$(\lambda_f\rho_h-\rho_h\lambda_f)_{i,j}=2(i-j)(\lambda_f)_{i,j}\stackrel{10.}{=}2(\lambda_f)_{i,j}$,
from where we get $(\lambda_f)_{i,j}=0$, if $j\neq i-1$. Also

\[(\lambda_f\rho_f)_{i,j} = \left\{
\begin{array}{l l}
-(\lambda_f)_{i,i-1}, & \quad \mbox{if $3\leq i\leq m+1$ and $j=i-2$}\\
0 & \quad \mbox{otherwise}\\
\end{array} \right. \]

\[(\rho_f\lambda_f)_{i,j} = \left\{
\begin{array}{l l}
-(\lambda_f)_{i-1,i-2}, & \quad \mbox{if $3\leq i\leq m+1$ and $j=i-2$}\\
0 & \quad \mbox{otherwise}\\
\end{array} \right. \]

By equation (12):\\
$(\lambda_f\rho_f)_{i,j}=(\rho_f\lambda_f)_{i,j}$, from where
$(\lambda_f)_{i,i-1}=(\lambda_f)_{j,j-1}$ $(2\leq i,j\leq m+1)$. That
means $\lambda_f=a\rho_f$. By the other term of equation (12),
$0=\rho_f\lambda_f+\lambda_f^2=(a+a^2)\rho_f$, and as $\rho_f\neq 0$,
we get $a+a^2=0$. That means $a=0$ or $a=-1$.

\begin{enumerate}
\item If $a=0$, then $\lambda_f=0$, so by equation (4) we get $\lambda_h=0$,
and by (7) we get $\lambda_e=0$.

\item If $a=-1$, so $\lambda_f=-\rho_f$, then
$\lambda_h\stackrel{4.}{=}\lambda_f\rho_e-\rho_e\lambda_f=\rho_e\rho_f-\rho_f\rho_e\stackrel{1.}{=}-\rho_h$.
Also
$\lambda_e\stackrel{7.}{=}\frac{1}{2}(\lambda_h\rho_e-\rho_e\lambda_h)=\frac{1}{2}(\rho_e\rho_h-\rho_h\rho_e)\stackrel{2.}{=}-\rho_e$.
We get that $\lambda=-\rho$.
\end{enumerate}

Summarizing: for every $m\in\mathbb{N}$, the algebra $\mathfrak{sl}_2$
has - up to equivalence -  exactly two irreducible Leibniz representations of
dimension $m+1$, each of these are Lie representations in the previous
sense. The two Leibniz-representations are as follows.
\begin{enumerate}
\item In appropriate basis, $\,\lambda =0$, and for  $\,1\leq i,j\leq m+1$,  
\[(\rho_e)_{i,j} = \left\{
\begin{array}{l l}
i(m+1-i), & \quad \mbox{if $j=i+1$}\\
0, & \quad \mbox{if $j\neq i+1$}\\
\end{array} \right., \]
\[(\rho_f)_{i,j} = \left\{
\begin{array}{l l}
-1, & \quad \mbox{if $j=i-1$}\\
0, & \quad \mbox{if $j\neq i-1$}\\
\end{array} \right., \]
\[(\rho_h)_{i,j} = \left\{
\begin{array}{l l}
m+2-2i, & \quad \mbox{if $j=i$}\\
0, & \quad \mbox{if $j\neq i$}\\
\end{array} \right.. \]

\item In appropriate basis, $\lambda=-\rho$, and for $(1\leq i,j\leq
m+1)$,
\[(\rho_e)_{i,j} = \left\{
\begin{array}{l l}
i(m+1-i), & \quad \mbox{if $j=i+1$}\\
0, & \quad \mbox{if $j\neq i+1$}\\
\end{array} \right. \]
\[(\rho_f)_{i,j} = \left\{
\begin{array}{l l}
-1, & \quad \mbox{if $j=i-1$}\\
0, & \quad \mbox{if $j\neq i-1$}\\
\end{array} \right. \]
\[(\rho_h)_{i,j} = \left\{
\begin{array}{l l}
m+2-2i, & \quad \mbox{if $j=i$}\\
0, & \quad \mbox{if $j\neq i$}\\
\end{array} \right. \]
\end{enumerate}

These two representations are obviously not equaivalent to each other.

We know that the extensions of dimension at least 5 of 
$\mathfrak{sl}_2$ can be obtained as follows.

Let $L$ be a simple $n$-dimensional Leibniz algebra $(n\geq5)$, for
which the Lie factor $L/I$ is isomorphic to
$\mathfrak{sl}_2$. We know that in this case there exists a basis
$\{e,f,h,x_0,x_1,\dots,x_{n-4}\}$ of $L$, in which the multiplication
table is as follows: (\cite{ORT})
$$\begin{array}{l}
[x_k,h]=(n-4-2k)x_k,\ (0\leq k\leq n-4)\\{}
[x_k,f]=x_{k+1},\ (0\leq k\leq n-5)\\{}
[x_k,e]=k(k+3-n)x_{k-1},\ (1\leq k\leq n-4)\\{}
[e,h]=-[h,e]=2e,\ [h,f]=-[f,h]=2f,\\{}
[e,f]=-[f,e]=h
\end{array}$$

It is easy to compute its representations. Let
$dim_{\mathbb{C}}(V)=m+1$, and search for possible homomorphisms
$\rho,\lambda:\rightarrow \mathfrak{gl}(V)$. We know that a
representation can be restricted to $\mathfrak{sl}_2$, and the
restriction of $\rho$ is irreducible. This way $\rho_e,\rho_f,\rho_h$
and $\lambda_e,\lambda_f,\lambda_h$ are given by the previous rules. We
would like to determine the images of the elements $x_i$. We can
distinguish two cases:

\begin{enumerate}
\item $n$ is odd. By the equality
$(n-4)\rho_{x_0}=\rho_{[x_0,h]}=\rho_h\rho_{x_0}-\rho_{x_0}\rho_h$ we
get 

\begin{flalign*}
(n-2)(\rho_{x_0})_{i,j} & =  2(j-i)(\rho_{x_0})_{i,j}\\
(n-2+2j-2i)(\rho_{x_0})_{i,j}& =  0.
\end{flalign*}

As $n$ is odd, $(n-2+2i-2j)\neq 0$, so $(\rho_{x_0})_{i,j}=0$ for
every pair $i,j$. We get $\rho_{x_0}=0$.\\
Using induction, we can see that $\rho_{x_k}=0$. For $k=0$ we have
this, and using the inductional assumption we get 

\begin{flalign*}
\rho_{x_{k+1}}=\rho_{[x_k,f]}=\rho_f\rho_{x_k}-\rho_{x_k}\rho_f=0-0=0.
\end{flalign*}

Apply now the above argument, replacing $\rho_{x_k}$ by
$\lambda_{x_k}$, and use $\lambda_{[x,y]}=\rho_y\lambda_x-\lambda_x\rho_y$.\\
We get that in this case $\rho_{x_k}=\lambda_{x_k}=0$, if $0\leq k\leq n-4$.

\item $n$ is even. Then let $n-4=2s$. In this case $[x_s,h]=(n-4-2s)x_s=0$.\\
From here by $0=\rho_{[x_s,h]}=\rho_h\rho_{x_s}-\rho_{x_s}\rho_h$ we
get that $\rho_{x_s}$ must be diagonal. Denote

\centerline{$(\rho_{x_s})_{i,i}=b_i$, $(1\leq i\leq m+1)$.}

On the other hand,
\begin{flalign*}
-s(s+1)\rho_{x_{s-1}}=\rho_{[x_s,e]}=\rho_e\rho_{x_s}-\rho_{x_s}\rho_e,
\end{flalign*}

which gives $(\rho_{x_{s-1}})_{i,j}=0$, if $j\neq i+1$, and

\begin{flalign*}
a_{i,i+1}=(\rho_{x_{s-1}})_{i,i+1}=\frac{(b_{i+1}-b_i)i(m+1-i)}{s(s+1)},
\end{flalign*}

if $1\leq i\leq m$. As $[x_{s-1},x_l]=0$, we get

\centerline{$0=\rho_{x_s}\rho_{x_{s-1}}-\rho_{x_{s-1}}\rho_{x_s}$,}

which means $(b_{i+1}-b_i)a_{i,i+1}=0$, if $1\leq i\leq m$.\\
So

\begin{flalign*}
0=(b_{i+1}-b_i)a_{i,i+1}=(b_{i+1}-b_i)\frac{(b_{i+1}-b_i)i(m+1-i)}{s(s+1)},
\end{flalign*}

from where $b_1=b_2=\dots=b_{m+1}$, and so $\rho_{x_{s-1}}=0$. For
$k\geq s$, using the product $[x_{k-1},f]=x_k$, by induction we get $\rho_{x_k}=0$.\\
For $k\leq s-1$ using the product $[x_{k},e]=k(k+3-n)x_{k-1}$ we get by
induction, that $k(k+3-n)\rho_{x_{k-1}}=0$, and as $1\leq k\leq s-1\leq
n-4$, so $k(k+3-n)\neq 0$, which means $\rho_{x_{k-1}}=0$.

We showed for every $k$ $(0\leq k \leq n-4)$ that $\rho_{x_k}=0$. 

For defining the values of $\lambda$, consider the expression
$0=\lambda_{[x_s,h]}=\rho_h\lambda_{x_s}-\lambda_{x_s}\rho_h$ . We get
that $\lambda_{x_s}$ is diagonal. Denote

\centerline{$(\lambda_{x_s})_{i,i}=c_i$, $(1\leq i\leq m+1)$.}

On the other hand,
\begin{flalign*}
-s(s+1)\lambda_{x_{s-1}}=\lambda_{[x_s,e]}=\rho_e\lambda_{x_s}-\lambda_{x_s}\rho_e,
\end{flalign*}

from which $(\lambda_{x_{s-1}})_{i,j}=0$, if $j\neq i+1$, and

\begin{flalign*}
d_{i,i+1}=(\lambda_{x_{s-1}})_{i,i+1}=\frac{(c_{i+1}-c_i)i(m+1-i)}{s(s+1)},
\end{flalign*}

if $1\leq i\leq m$. Using $[x_{s-1},x_l]=0$, we get

\begin{flalign*}
 0 &=\rho_{x_s}\lambda_{x_{s-1}}+\lambda_{x_{s-1}}\lambda_{x_s} =\lambda_{x_{s-1}}\lambda_{x_s}=\\
 &=\rho_{x_{s-1}}\lambda_{x_{s}}+\lambda_{x_{s}}\lambda_{x_{s-1}} =\lambda_{x_{s}}\lambda_{x_{s-1}}\\
\end{flalign*}

This means
$0=\lambda_{x_{s-1}}\lambda_{x_s}-\lambda_{x_{s}}\lambda_{x_{s-1}}$ and

\begin{flalign*}
0=(c_{i+1}-c_i)d_{i,i+1}=(c_{i+1}-c_i)\frac{(c_{i+1}-c_i)i(m+1-i)}{s(s+1)}.
\end{flalign*}

With the previous argument we get $\lambda_{x_k}=0$, if $0\leq k\leq n-4$.
\end{enumerate}

In both cases the final result is that for the elements
$\{x_0,\dots,x_{n-4}\}$, $\rho$ and $\lambda$ is zero. This means that
the Leibniz algebra $L$ has for every $m$ two types of
$(m+1)$-dimensional Leibniz representations, such that $\rho$,
restricted to the subalgebra $\mathfrak{sl}_2$ is irreducible. We get these two representations from representations of
$\mathfrak{sl}_2$ by chosing $\rho|_I=\lambda|_I=0$ (here $I$ denotes
the Leibniz kernel (generated by the elements $x_k^2$).

\subsection{General case}
\smallskip

Using the results from the previous section, consider irreducible
representations of a semisimple Leibniz algebra $L$. We know that
$L=S+I$ as vector space. As $I$ is generated by the squares of elements
in $L$, for every $x\in I$ there exists $n\in \mathbb{N}$,
$\alpha_1,\dots,\alpha_n\in K$, and $x_1,\dots,x_n\in L$ such that
$x=\alpha_1[x_1,x_1]+\dots+\alpha_n[x_n,x_n]$. Using the linearity of
$\rho$ and the identity $\rho_{[x,y]}=\rho_y\rho_x-\rho_x\rho_y$ we
get:

$\rho_x=\alpha_1\rho_{[x_1,x_1]}+\dots+\alpha_n\rho_{[x_n,x_n]}$

$\rho_{[x_i,x_i]}=\rho_{x_i}\rho_{x_i}-\rho_{x_i}\rho_{x_i}=0$, so
$\rho_x=0$ for every $x\in I$.

We have $L=S+I$ as vector space, so every $y\in L$ can be uniquely
written in the form $y=s+x$, where $s\in S$ and $x\in I$. We get
$\rho_y=\rho_s+\rho_x=\rho_s$, which means $\rho(L)=\rho(S)$. We
assumed that $\rho$ is irreducible, so if a subspace $U$ in $V$ is invariant under every $a\in
\rho(L)$, then $U=0$ or $U=V$. By the above argument, it is satisfied
for a semisimple Leibniz algebra if
and only if the restriction of $\rho$ to $S$, the map
$(-\rho)|_S:S\rightarrow \mathfrak{gl}(V)$, as a representation of the
semisimple Lie algebra $S$, is irreducible.

If this is satisfied, then by \ref{irred_lie} we get $\lambda=-\rho$ or $\lambda=0$. In any case, as
$\rho|_I=0$, we must have $\lambda|_I=0$. Here $(-\rho)|_S$ is an
irreducible representation of the semisimple Lie algebra $S$.

Let us summarize the results.

\begin{thm}
Let $L$ be a semisimple Leibniz algebra over the field $K$, with
$char(K)=0$. Let $V$ be a vectorspace over $K$. Then $L$ can be written
in the form $L=S\dot{+}I$, where $I$ is the Leibniz kernel, $S$ is a
semisimple Lie algebra which is a subalgebra in $L$. Then every
irreducible representation $\rho,\lambda:L\rightarrow \mathfrak{gl}(V)$
of $L$ can be written in the following form:

$\rho|_I=\lambda|_I=0$, $\rho|_S=-\varphi$, where $\varphi:
S\rightarrow \mathfrak{gl}(V)$ is an irreducible representation of the
semisimple Lie algebra $S$, and $\lambda|_S=(-\rho)|_S$ or $\lambda|_S=0$.
\end{thm}

That means that every irreducible representation of $L$ can be obtained
from the irreducible representations of $S$, and we do not get new, not
Lie type representations.

\section{Reducible representations of semisimple Leibniz algebras}

For semisimple Lie algebras we know that every representation is
completely  reducible, that means it can be written as a direct sum of
irreducible ones. The question is the following. Can this statement be
carried over to Leibniz algebras?

If a Leibniz representation splits into the sum of irreducible ones, i.e.
{$V=V_1\oplus\dots\oplus V_k$}, then for every $i$, ($1\leq i\leq k$)
$\lambda|_{V_i}=0$ or $\lambda|_{V_i}=(-\rho)|_{V_i}$. As for $x\in I$
($I$ is the Leibniz kernel) $\rho_x=0$, then in both cases,
$\lambda_x|_{V_i}$=0. As $V=V_1\oplus\dots\oplus V_k$, we get $\lambda_x=0$.

That means if a representation decomposes into a direct sum of
irreducible ones, then we must have $\rho|_I=\lambda|_I=0$.

\begin{exa}
Consider the following simple Leibniz algebra: $L=span\{e,f,h,x,y\}$,
the nontrivial Leibniz brackets are:

$$\begin{array}{l}
[e,f]=-[f,e]=h, [e,h]=-[h,e]=2e, [f,h]=-[h,f]=-2f,\\{}
[x,h]=-[y,e]=x, [x,f]=-[y,h]=y.
\end{array}$$

For this algebra $I=span\{x,y\}$. Let $V:=L$, so consider $L$ as $L$-module: for $v,z\in L$,
$\rho_z(v)=[v,z]$ and $\lambda_z(v)=[z,v]$.

It is easy to see from the bracket table that $\lambda|_I\neq 0$ and
the representation can not be decomposed into irreducible ones. Even
more, in this case the only nontrivial invariant subspace is $U=I$,
so we indeed did not find any appropriate decomposition. The fact that
$I$ is the only nontrivial invariant subspace, follows from the
simplicity: the invariant subspaces are in this case exactly the
ideals.
\end{exa}

We get the following 
\begin{thm}
For Leibniz algebras it is not true that a representation of a (semi)simple Leibniz
 algebra always decomposes into the direct sum of irreducible ones.
 \end{thm}
\medskip
\begin{rem}
Of course, there are cases when the representation can be decomposed into
irreducible ones.
\end{rem}

\begin{exa}
Consider a five-dimensional, \emph{not} irreducible representation of
$\mathfrak{sl}_2$:

$
\rho_e=\left(
\begin{array}{ccc|cc}
0&2&0&0&0\\
0&0&2&0&0\\
0&0&0&0&0\\ \hline
0&0&0&0&1\\
0&0&0&0&0
\end{array}\right)
$
\qquad
$\rho_f=\left(
\begin{array}{ccc|cc}
0&0&0&0&0\\
-1&0&0&0&0\\
0&-1&0&0&0\\ \hline
0&0&0&0&0\\
0&0&0&-1&0
\end{array}\right)
$

$
\rho_h=\left(
\begin{array}{ccc|cc}
2&0&0&0&0\\
0&0&0&0&0\\
0&0&-2&0&0\\ \hline
0&0&0&1&0\\
0&0&0&0&-1
\end{array}\right)
$

Again, starting with the identity (10), we get

$\lambda_f=\left(
\begin{array}{ccc|cc}
0&0&0&0&0\\
a_{21}&0&0&0&0\\
0&a_{32}&0&0&0\\ \hline
0&0&0&0&0\\
0&0&0&a_{54}&0
\end{array}\right)$, so $\lambda_f$ is block diagonal.

From the identity (4), $\lambda_h$, from (7), $\lambda_e$ are also block
diagonal, which means this representation can be written as a direct
sum of irreducible ones. The whole point was, that $\lambda_f$ had
block diagonal form. This happened because $(\lambda_f)_{i,j}$=0, if $(\rho_h)_{j,j}-(\rho_h)_{i,i}\neq 2$.

For $\lambda_f$ not (necessarily) being block diagonal, one needs 
$(\rho_h)_{j,j}-(\rho_h)_{i,i}=2$ for such pairs $\{i,j\}$,
where $i$ and $j$ are indices of rows which do not belong to the
same block along the diagonal.

If for instance the representation has dimension $(2n+1)$, and we try
to decompose it to the sum of two irreducible ones, then one of them
has dimension $2m$, in the appropriate block there are odd numbers
along the main diagonal of $\rho_h$, while the other one has dimension
$(2n-2m+1)$, and in the appropriate block there are even numbers along
the main diagonal of $\rho_h$. So the condition
$(\rho_h)_{j,j}-(\rho_h)_{i,i}=2$ can only be satisfied, if $i$ and
$j$ are indeces of such rows, which belong to the same block. In this
case everything is diagonal, and so the representation can be
decomposed into the sum of irreducible ones.
\end{exa}

\end{document}